\documentclass[12pt]{article}
\usepackage{mathrsfs}
\usepackage{bbm}

\usepackage{amssymb,a4}
\usepackage{amsmath,amsfonts,amssymb,amsthm}

\newcommand{\1}{\mathbbm {1}}
\newcommand{\Z}{{\mathbb Z}}
\newcommand{\C}{{\mathbb C}}
\newcommand{\h}{{\mathfrak h}}
\newcommand{\wh}{{\widehat{\mathfrak h}}}
\newcommand{\I}{{\mathcal I}}
\newcommand{\tI}{\widetilde{\I}}
\newcommand{\J}{{\mathcal J}}

\def\<{\langle}
\def\>{\rangle}
\def\a{\alpha}
\def\b{\beta}
\def\g{\mathfrak g}
\def\h{\mathfrak h}

\newcommand{\mraff}{\mathrm{aff}}

\newcommand{\la}{\langle}
\newcommand{\ra}{\rangle}

\newtheorem{thm}{Theorem}[section]
\newtheorem{prop}[thm]{Proposition}
\newtheorem{lem}[thm]{Lemma}

\newtheorem{rmk}[thm]{Remark}

\begin{document}

\begin{center}
{\Large \bf  The structure of parafermion vertex operator algebras $K(osp(1|2n),k)$}

\end{center}

\begin{center}
{ Cuipo Jiang$^{a}$\footnote{Supported by China NSF grant No.11771281.}
and Qing Wang$^{b}$\footnote{Supported by
China NSF grants No.12071385 and the Fundamental Research Funds for the Central Universities No.20720200067.}\\
$\mbox{}^{a}$ School of Mathematical Sciences, Shanghai Jiao Tong University, Shanghai 200240, China\\
\vspace{.1cm}
$\mbox{}^{b}$ School of Mathematical Sciences, Xiamen University,
Xiamen 361005, China\\
}
\end{center}


\begin{abstract}
In this paper, the structure of the parafermion vertex operator algebra associated to an
integrable highest weight module for simple affine Lie superalgebra $osp(1|2n)$
is studied. Particularly, we determine the generators for this algebra.
\end{abstract}


\section{Introduction}
\def\theequation{1.\arabic{equation}}
\setcounter{equation}{0}
Let $\g$ be a simple Lie superalgebra and let $L_{\hat{\g}}(k,0)$ be the simple affine vertex superalgebra associated to the affine Lie superalgebra $\hat{\g}$ with the level $k$. If $\g$ is a Lie algebra, $L_{\hat{\g}}(k,0)$ is a $C_2$-cofinite and rational vertex operator algebra if and only if $k$ is a positive integer \cite{FZ}, \cite{DL}, \cite{Li}. If $\g$ is not a Lie algebra, Gorelik and Kac \cite{GK} claimed that $L_{\hat{\g}}(k,0)$ is $C_2$-cofinite if and only if $\g$ is the simple Lie superalgebra $osp(1|2n)$ and $k$ is a positive integer, which was proved recently in \cite{AL} and \cite{CL}. Also in \cite{CL}, Creutzig and Linshaw proved the rationality of the affine vertex operator superalgebra $L_{\hat{\g}}(k,0)$ with $k$ being a positive integer. The structural and representation theory of the rational parafermion vertex operator algebras associated to the integrable highest weight modules of affine Kac-Moody Lie algebras and their orbifolds have been fully studied (see \cite{ADJR,ALY1,ALY2,DLY2,DLWY,DR,DW1,DW2,DW3,JW1,JW2,Lam,Wang} etc.) In this paper, we turn our attention to the rational parafermion vertex operator algebras associated to the affine vertex superalgebras. We study the structure of the rational parafermion vertex operator algebra $K(\g,k)$ associated to the simple Lie superalgebra $\g=osp(1|2n)$ with $k$ a positive integer. Specifically, we determine the generators of the parafermion vertex operator algebras $K(osp(1|2n),k)$. The generator result shows that the parafermion vertex operator algebra $K(osp(1|2),k)$ associated to $osp(1|2)$ together with $K(sl_2,2k)$ associated to $sl_2$ are building blocks of $K(osp(1|2n),k)$. The structural and
representation theories for $K(sl_2,k)$ are studied in \cite{DLWY}, \cite{DW2}, \cite{ALY}, \cite{JW1}, \cite{JW2} etc. And the
representation theory for $K(osp(1|2),k)$ are studied in \cite{CFK}. These may shed light on the study of
representation theory for $K(osp(1|2n),k)$.

  Let $\g=osp(1|2n)$ and  $\h$ be its Cartan subalgebra, it is known that $L_{\hat{\g}}(k,0)$ is the simple quotient of the universal vacuum module $V_{\hat{\g}}(k,0)$. The maximal submodule of $V_{\hat{\g}}(k,0)$ is generated by $e_{\theta}(-1)^{k+1}\1$ \cite{GS}, where $\theta$ is the highest root of $\g$. As for the parafermion vertex operator algebra, let $M_{\hat{\h}}(k,0)$ be the Heisenberg vertex subalgebra of $V_{\hat{\g}}(k,0)$,  and $K(\g,k)$  the simple quotient of the commutant vertex operator algebra $N(\g,k)=$Com$(M_{\hat{\h}}(k,0),V_{\hat{\g}}(k,0))$. We determine the generators of the maximal ideal of $N(\g,k)$ and further characterize the structure of the parafermion vertex operator algebra $K(\g,k)$.

  The paper is organized as follows. In Section 2, we recall the construction of the vertex operator superalgebras $V(k,0)$ associated to the simple Lie superalgebra $osp(1|2n)$. Let $V(k,0)(0)=\{v\in V(k,0)|h(0)v=0, \forall h\in \h\}$ be the subalgebra of $V(k,0)$, where $\h$ is the Cartan subalgebra of $osp(1|2n)$. Since
  $V(k,0)(0)=M_{\hat{\h}}(k,0)\otimes N(osp(1|2n),k)$, we first give the generators of the vertex operator algebra $V(k,0)(0)$ in this section. Then we determine the generators of the commutant vertex operator algebra $N(osp(1|2n),k)$ and prove that $N(osp(1|2),k)$ together with $N(sl_2,2k)$ are the building block of $N(osp(1|2n),k)$ in Section 3. In Section 4, we give a set of generators for the parafermion vertex operator algebra $K(osp(1|2n),k)$, which is the simple quotient of $N(osp(1|2n),k)$. We also give the generator of the maximal ideal of $N(osp(1|2n),k)$.

\section{Vertex operator superalgebras $V(k,0)$ and vertex operator subalgebras $V(k,0)(0)$}
\label{Sect:V(k,0)}
\def\theequation{2.\arabic{equation}}
\setcounter{equation}{0}

Let $\g$ be the finite dimensional simple Lie superalgebra $osp(1|2n)$ with a Cartan
subalgebra $\h.$ Let $\Delta$ be the corresponding root system, $\Delta_{0}$ the root system of even and $\Delta_{1}$ the root system of odd, $\Delta_{0}^{L}$ the set of long roots in $\Delta_{0}$ and $\Delta_{0}^{S}$ the set of short roots in $\Delta_{0}$, and
$Q$ the root lattice. Let  $\la ,\ra$ be an invariant even supersymmetric
nondegenerate bilinear form on $\g$ such that $\<\a,\a\>=2$ if
$\alpha$ is a long root in $\Delta_{0}$, where we have identified $\h$ with $\h^*$
via $\<,\>.$ As in \cite{H}, we denote the image of $\alpha\in
\h^*$ in $\h$ by $t_\alpha.$ That is, $\alpha(h)=\<t_\alpha,h\>$
for any $h\in\h.$ Fix simple roots $\{\alpha_1,\cdots,\alpha_n\}$
and denote the highest root by $\theta.$

Let $\g_{\alpha}$ denote the root space associated to the root
$\a\in \Delta.$ For $\alpha\in \Delta_{0(+)}^{S}$, we fix $e_{\pm
\alpha}\in \g_{\pm \alpha}$ and
$h_{\alpha}=2t_\alpha\in \h$ such that
 $[e_\a,e_{-\a}]=h_{\a}, [h_\a,e_{\pm \a}]=\pm 2e_{\pm\a}.$ That
is, $\g^{\a}=\C e_{\a}+\C h_{\alpha}+\C e_{-\alpha}$ is isomorphic
to $sl_2$. For $\alpha\in \Delta_{0(+)}^{L}$, we fix $e_{\pm
\alpha}\in \g_{\pm \alpha}$, $x_{\pm\frac{1}{2} \alpha}\in \g_{\pm \frac{1}{2}\alpha}$, $h_{\alpha}=t_\alpha\in \h$ such that
$$[e_{\alpha}, e_{-\alpha}]=h_{\alpha}, [h_{\alpha},e_{\pm\alpha}]=\pm 2e_{\pm\alpha}, [h_{\alpha}, x_{\pm \frac{1}{2}\alpha}]=\pm x_{\pm\frac{1}{2}\alpha},$$

$$ [e_{\alpha},x_{\frac{1}{2}\alpha}]=0, [e_{-\alpha},x_{\frac{1}{2}\alpha}]=-x_{-\frac{1}{2}\alpha}, [e_{\alpha},x_{-\frac{1}{2}\alpha}]=-x_{\frac{1}{2}\alpha}, [e_{-\alpha},x_{-\frac{1}{2}\alpha}]=0,$$

$$\{x_{\frac{1}{2}\alpha}, x_{\frac{1}{2}\alpha}\}=2e_{\alpha}, \{x_{\frac{1}{2}\alpha},x_{-\frac{1}{2}\alpha}\}=h_{\alpha}, \{x_{-\frac{1}{2}\alpha},x_{-\frac{1}{2}\alpha}\}=-2e_{-\alpha},$$
where we notice that $\pm \frac{1}{2}\alpha\in \Delta_{1}$. That is, $\bar{\g}^{\alpha}=\C e_{\a}+\C h_{\alpha}+\C e_{-\alpha}+\C x_{\a}+\C x_{-\a}$
 is isomorphic
to $osp(1|2)$.
 Then
$\<h_\a,h_\a\>=\frac{4}{\<\alpha,\alpha\>}$ and
$\<e_{\a},e_{-\a}\>=\frac{2}{\<\alpha,\alpha\>}$
for all
$\alpha\in \Delta_{0},$ and $\<\a,\a\>=2$ and
$\<x_{\frac{1}{2}\a},x_{-\frac{1}{2}\a}\>=-\<x_{-\frac{1}{2}\a},x_{\frac{1}{2}\a}\>=2$
for
$\alpha\in \Delta_{0}^{L}.$ For more informations on the root decompositions and root systems of the Lie superalgebra $\g$ can refer to \cite{Kac}.

Let $\widehat{\mathfrak g}= \g \otimes \C[t,t^{-1}] \oplus \C K$
be the corresponding affine Lie superalgebra. Let $k$ be a positive
integer and
\begin{equation*}
V(k,0) = V_{\widehat{\g}}(k,0) = Ind_{\g \otimes
\C[t]\oplus \C K}^{\widehat{\g}}\C
\end{equation*}
the induced $\widehat{\g}$-module such that ${\g} \otimes \C[t]$ acts as $0$ and $K$ acts as $k$ on $\1=1$.

We denote by $a(n)$ the operator on $V(k,0)$ corresponding to the action of
$a \otimes t^n$. Then
$$[a(m), b(n)] = [a,b](m+n) + m \la a,b \ra \delta_{m+n,0}k$$
for $a, b \in \g$ and $m,n\in \Z$.

Let $a(z) = \sum_{n \in \Z} a(n)z^{-n-1}$. Then $V(k,0)$ is a
vertex operator superalgebra generated by $a(-1)\1$ for $a\in \g$ such
that $Y(a(-1)\1,z) = a(z)$ with the
vacuum vector $\1$ and the Virasoro vector
\begin{align*}
\omega_{\mraff} &= \frac{1}{2(k+n+\frac{1}{2})} \Big(
\sum_{i=1}^{n}h_i(-1)h_i(-1)\1 +\sum_{ \alpha\in\Delta_{0}}
\frac{\<\a,\a\>}{2}e_{\alpha}(-1)e_{-\alpha}(-1)\1
\\&-\sum_{ \alpha\in\Delta_{1(+)}}
\frac{1}{2}x_{\alpha}(-1)x_{-\alpha}(-1)\1+\sum_{ \alpha\in\Delta_{1(+)}}
\frac{1}{2}x_{-\alpha}(-1)x_{\alpha}(-1)\1
\Big)
\end{align*}
of central charge $\frac{kn(2n-1)}{k+n+\frac{1}{2}}$ (e.g.
\cite{KRW}), where $h^{\vee}$ is the dual
Coxeter number of $\g$, $\{h_i|i=1,\cdots,n\}$ is an
orthonormal basis of $\mathfrak h,$ $\Delta_{1(+)}$ is the set of the positive odd roots.

We will use the standard notation for the component operators of
$Y(u,z)$ for $u\in V(k,0).$ That is, $Y(u,z)=\sum_{n\in \Z}u_nz^{-n-1}.$ From the
definition of vertex operators, we see that
$(a(-1)\1)_n=a(n)$ for $a\in \g.$ So in the rest of this paper, we will use
both $a(n)$ and $(a(-1)\1)_n$ for $a\in \g$ and use $u_n$ only for general $u$
without further explanation.

For $\lambda \in {\mathfrak h}^*$, set
\begin{equation*}
V(k,0)(\lambda)=\{v\in V(k,0)|h(0)v=\lambda(h) v, \forall\;
h\in\mathfrak h\}.
\end{equation*} Then we have
\begin{equation}\label{eq:V-dec}
V(k,0)=\oplus_{\lambda\in Q}V(k,0)(\lambda).
\end{equation}

Since $[h(0), Y(u,z)]=Y(h(0)u,z)$ for $h\in \mathfrak h$ and  $u
\in V(k,0)$, from the definition of affine vertex operator
superalgebra, we see that $V(k,0)(0)$ is a vertex operator subalgebra
of $V(k,0)$ with the same Virasoro vector $\omega_{\mraff}$ and
each $V(k,0)(\lambda)$ is a module for $V(k,0)(0)$.

The first theorem is on generators for $V(k,0)(0).$

\begin{thm}\label{generator1} The vertex operator algebra
$V(k,0)(0)$ is generated by vectors $\alpha_{i}(-1)\1$ and $e_{-\alpha}(-2)e_{\alpha}(-1)\1$, $x_{-\frac{1}{2}\alpha}(-2)x_{\frac{1}{2}\alpha}(-1)\1$, $e_{-\beta}(-2)e_{\beta}(-1)\1$ for
$1\leq i \leq n, \alpha\in\Delta_{0(+)}^{L}, \beta\in\Delta_{0(+)}^{S}$, where $\Delta_{0(+)}^{L}$ and $\Delta_{0(+)}^{S}$ are the sets of even positive long roots and even positive short roots respectively.
\end{thm}

\begin{proof} First note that $V(k,0)(0)$ is spanned
by the vectors
$$a_1(-m_1)\cdots a_s(-m_s)e_{\beta_{1}}(-n_{1})e_{\beta_{2}}(-n_{2})\cdots
e_{\beta_{\nu}}(-n_{\nu})x_{\beta_{\nu+1}}(-n_{\nu+1})x_{\beta_{\nu+2}}(-n_{\nu+2})\cdots
x_{\beta_{t}}(-n_{t})\1$$
where $a_i\in \mathfrak h, \beta_j\in\Delta, m_i>0, n_j>0$ and
$\beta_{1}+\beta_{2}+\cdots+\beta_{t}=0.$ Let $U$ be the vertex
operator subalgebra generated by $\alpha_{i}(-1)\1$ and $e_{-\alpha}(-2)e_{\alpha}(-1)\1$, $x_{-\frac{1}{2}\alpha}(-2)x_{\frac{1}{2}\alpha}(-1)\1$, $e_{-\beta}(-2)e_{\beta}(-1)\1$ for
$1\leq i \leq n, \alpha\in\Delta_{0(+)}^L, \beta\in\Delta_{0(+)}^{S}$. Clearly, $\alpha_{i}(-1)\1$ and
$e_{-\alpha}(-2)e_{\alpha}(-1)\1,$ $x_{-\frac{1}{2}\alpha}(-2)x_{\frac{1}{2}\alpha}(-1)\1, e_{-\beta}(-2)e_{\beta}(-1)\1\in V(k,0)(0)$ for $1\leq i \leq n, \alpha\in\Delta_{0(+)}^{L}, \beta\in\Delta_{0(+)}^{S}$. Thus, it suffices to prove that
$V(k,0)(0)\subset U.$

Since $(h(-1)\1)_n = h(n)$ for $h\in \mathfrak h $, we see that
$h(n)U \subset U$ for $h\in \mathfrak h$ and $n\in\Z.$ So we only
need to prove $$u=e_{\beta_{1}}(-n_{1})e_{\beta_{2}}(-n_{2})\cdots
e_{\beta_{\nu}}(-n_{\nu})x_{\beta_{\nu+1}}(-n_{\nu+1})x_{\beta_{\nu+2}}(-n_{\nu+2})\cdots
x_{\beta_{t}}(-n_{t})\1\in U$$ with
$\beta_{1}+\beta_{2}+\cdots+\beta_{t}=0$.  We will prove by
induction on $t$.

It is obvious that $t\geq 2.$ If $t=2,$ it follows from Theorem 2.1 in
\cite{DLWY} that
$$e_{-\alpha}(-m)e_{\alpha}(-n)\1 \in U, \ x_{-\beta}(-m)x_{\beta}(-n)\1 \in U$$
 for $m,n>0, \alpha\in \Delta_{0}, \beta\in \Delta_{1}$.
Note that if $m\geq 0$, then  $$e_{-\alpha}(m)e_{\alpha}(n)\1=
-h_{\alpha}(m+n)\1+mk\<e_{\alpha},e_{-\alpha}\>\delta_{m+n,0}\1\in
U,$$
$$x_{-\beta}(m)x_{\beta}(n)\1=
h_{2\beta}(m+n)\1+mk\<x_{-\beta},x_{\beta}\>\delta_{m+n,0}\1\in
U.$$

We claim that $e_{-\alpha}(m)e_{\alpha}(n)U\subset U$, $x_{-\beta}(m)x_{\beta}(n)U\subset U$ for all
$m,n\in \Z, \alpha\in \Delta_{0}, \beta\in \Delta_{1}.$ Let $u\in U.$ From Proposition 4.5.7 of \cite{LL},
there exist nonnegative integers $p,q$ such that
$$e_{-\alpha}(m)e_{\alpha}(n)u=\sum_{i=0}^p\sum_{j=0}^q\binom{m-q}{i}\binom{q}{j}(e_{-\alpha}(m-q-i+j)e_{\alpha}(-1)\1)_{n+q+i-j}u,$$
and there exist nonnegative integers $p^{'},q^{'}$ such that
$$x_{-\beta}(m)x_{\beta}(n)u=\sum_{i=0}^{p_{'}}\sum_{j=0}^{q^{'}}\binom{m-q^{'}}{i}\binom{q^{'}}{j}(x_{-\beta}(m-q^{'}-i+j)x_{\beta}(-1)\1)_{n+q^{'}+i-j}u.$$
Since $e_{-\alpha}(m-q-i+j)e_{\alpha}(-1)\1\in U, x_{-\beta}(m-q^{'}-i+j)x_{\beta}(-1)\1\in U$, the claim
follows.

Next we assume that $t>2$ and that
$$e_{\beta_{1}}(-n_{1})e_{\beta_{2}}(-n_{2})\cdots
e_{\beta_{\nu}}(-n_{\nu})x_{\beta_{\nu+1}}(-n_{\nu+1})x_{\beta_{\nu+2}}(-n_{\nu+2})\cdots
x_{\beta_{p}}(-n_{p})\1\in U$$ with
$\beta_{1}+\beta_{2}+\cdots+\beta_{p}=0$ for $2\leq p \leq
t-1$ and $n_i>0.$ We have to show that
$$e_{\beta_{1}}(-n_{1})e_{\beta_{2}}(-n_{2})\cdots
e_{\beta_{\nu}}(-n_{\nu})x_{\beta_{\nu+1}}(-n_{\nu+1})x_{\beta_{\nu+2}}(-n_{\nu+2})\cdots
x_{\beta_{t}}(-n_{t})\1\in U$$ with
$\beta_{1}+\beta_{2}+\cdots+\beta_{t}=0$. We divide the proof into
two cases.

{\bf Case 1.} There exist $1\leq i,j\leq t$ such that
$\beta_{i}+\beta_{j}\in \Delta.$ Note that if
 $$e_{\beta_{1}}(-n_{1})e_{\beta_{2}}(-n_{2})\cdots
e_{\beta_{\nu}}(-n_{\nu})x_{\beta_{\nu+1}}(-n_{\nu+1})x_{\beta_{\nu+2}}(-n_{\nu+2})\cdots
x_{\beta_{t}}(-n_{t})\1\in U,$$
 then we consider the following three subcases: $(1) 1\leq i,j\leq \nu$, $(2) 1\leq i\leq \nu, \ \nu+1\leq j\leq t$, $(3) \nu+1\leq i, j\leq t$.\\

If $(1) 1\leq i,j\leq \nu$, then $$e_{\beta_{i_1}}(-n_{i_1})\cdots e_{\beta_{i_\nu}}(-n_{i_\nu})x_{\beta_{i_{\nu+1}}}(-n_{i_{\nu+1}})\cdots
x_{\beta_{i_t}}(-n_{i_t})\1\in U$$ by the induction assumption,
where $(i_1,...,i_\nu)$ is any permutation of $(1,...,\nu)$ and $(i_{\nu+1},...,i_t)$ is any permutation of $(\nu+1,...,t).$ Without
loss of generality, we may assume that $\beta_{1}+\beta_{2}\in
\Delta$.

Let $m,n$ be positive integers such that $-m+n=-n_2$ and $n>n_i$
for $i\geq 3.$ Let
$w=e_{\beta_{1}+\beta_{2}}(-m)e_{\beta_{3}}(-n_{3})\cdots
\cdots
e_{\beta_{\nu}}(-n_{\nu})x_{\beta_{\nu+1}}(-n_{\nu+1})x_{\beta_{\nu+2}}(-n_{\nu+2})\cdots
x_{\beta_{t}}(-n_{t})\1$ with
$\beta_{1}+\beta_{2}+\cdots+\beta_{t}=0.$ Then $w\in U$ by the
induction assumption and $e_{\beta_1}(-n_1)e_{-{\beta_1}}(n)w\in
U$ by the claim.

Let $[e_{-\b_1},e_{\b_1+\b_2}]=\lambda e_{\b_2}$ for some nonzero
$\lambda.$ Then
\begin{equation*}
\begin{split}
&\qquad \quad e_{\beta_1}(-n_1)e_{-\beta_1}(n)w= \lambda
e_{\beta_1}(-n_1)e_{\beta_2}(-n_2)e_{\beta_3}(-n_3)\cdots
e_{\beta_t}(-n_t)\1 \\
& \qquad \qquad +
e_{\b_1}(-n_1)e_{\b_1+\b_2}(-m)[e_{-\b_1},e_{\b_3}](n-n_3)e_{\b_4}(-n_4)\cdots
e_{\beta_t}(-n_t)\1\\
& \qquad \qquad +\cdots
+e_{\b_1}(-n_1)e_{\b_1+\b_2}(-m)e_{\b_3}(-n_3)\cdots
[e_{-\beta_1},x_{\beta_t}](n-n_t)\1.
\end{split}
\end{equation*}
Since $n-n_i>0$ for $i\geq 3$, we see that
\begin{equation*}
\begin{split}
& \qquad
e_{\b_1}(-n_1)e_{\b_1+\b_2}(-m)[e_{-\b_1},e_{\b_3}](n-n_3)e_{\b_4}(-n_4)\cdots
e_{\beta_t}(-n_t)\1\\
& \qquad \qquad +\cdots
+e_{\b_1}(-n_1)e_{\b_1+\b_2}(-m)e_{\b_3}(-n_3)\cdots
[e_{-\beta_1},x_{\beta_t}](n-n_t)\1
\end{split}
\end{equation*}
lies in $U$ by induction assumption. As a result,
$$e_{\beta_{1}}(-n_{1})e_{\beta_{2}}(-n_{2})\cdots
e_{\beta_{\nu}}(-n_{\nu})x_{\beta_{\nu+1}}(-n_{\nu+1})x_{\beta_{\nu+2}}(-n_{\nu+2})\cdots
x_{\beta_{t}}(-n_{t})\1\in U.$$ Similar analysis to the subcases (2) and (3), we can also get $$e_{\beta_{1}}(-n_{1})e_{\beta_{2}}(-n_{2})\cdots
e_{\beta_{\nu}}(-n_{\nu})x_{\beta_{\nu+1}}(-n_{\nu+1})x_{\beta_{\nu+2}}(-n_{\nu+2})\cdots
x_{\beta_{t}}(-n_{t})\1\in U.$$

{\bf Case 2.} For any $1\leq i,j\leq t,
\;\beta_{i}+\beta_{j}\notin \Delta.$ We claim that  there exist
$1\leq i^{'},j^{'}\leq t$ such that
$\beta_{i^{'}}+\beta_{j^{'}}=0.$ In fact, if $\b_i+\b_j\ne 0$ for
all $i,j.$ Together with the fact that for $\alpha,\beta\in\Delta_{0(+)}^L$, $\alpha+\beta\notin\Delta_{0(+)}^L
$, we then deduce that $\<\b_i,\b_j\>\geq 0$ for all $i,j$.
Thus $\<\b_1,\sum_{j=2}^t\b_j\>\geq 0.$ On the other hand, since
$\sum_{j=2}^t\b_j=-\b_1$, notice that $\g=osp(1|2n)$, we have $\<\b_1,\sum_{j=2}^t\b_j\><0,$ a
contradiction. Without loss of generality, we may assume $\beta_{1}+\beta_{2}=0$ or $\beta_{\nu+1}+\beta_{\nu+2}=0$. Then
$\beta_3+\cdots+\beta_t=0$ or $\beta_1+\cdots+\beta_\nu+\beta_{\nu+3}+\cdots+\beta_t=0$. By the induction assumption,
$$e_{\beta_{3}}(-n_{3})\cdots e_{\beta_{\nu}}(-n_{\nu})\cdot x_{\beta_{\nu+1}}(-n_{\nu+1}) \cdots
x_{\beta_{t}}(-n_{t})\1\in U,$$  or $$e_{\beta_{1}}(-n_{1})\cdots
e_{\beta_{\nu}}(-n_{\nu})x_{\beta_{\nu+3}}(-n_{\nu+3})\cdots
x_{\beta_{t}}(-n_{t})\1\in U.$$
Thus from the above claim, we have
$$e_{\beta_{1}}(-n_{1})e_{\beta_{2}}(-n_{2})\cdots
e_{\beta_{\nu}}(-n_{\nu})x_{\beta_{\nu+1}}(-n_{\nu+1})x_{\beta_{\nu+2}}(-n_{\nu+2})\cdots
x_{\beta_{t}}(-n_{t})\1\in U$$ as desired.
\end{proof}


\section{Vertex operator algebra $N(\g,k)$ }
\label{walgebra}
\def\theequation{3.\arabic{equation}}
\setcounter{equation}{0}

Let $V_{\wh}(k,0)$ be the vertex operator subalgebra of $V(k,0)$
generated by $h(-1)\1$ for $h\in \mathfrak h$ with the Virasoro
element
$$\omega_{\mathfrak h} = \frac{1}{2k}
\sum_{i=1}^{n}h_i(-1)h_i(-1)\1$$
of central charge $n$, where $\{h_1,\cdots h_n\}$ is an
orthonormal basis of $\mathfrak h$ as before. For $\lambda\in
{\mathfrak h}^*$, let $M_{\wh}(k,\lambda)$ denote the irreducible
highest weight module for $\wh$ with a highest weight vector
$v_\lambda$ such that $h(0)v_\lambda = \lambda(h) v_\lambda$ for
$h\in \mathfrak h.$ Then $V_{\wh}(k,0)$ is identified with
$M_{\wh}(k,0).$

Recall from  Section 2 that both $V(k,0)$ and $V(k,0)(\lambda)$, $\lambda \in Q$
are completely reducible $V_{\wh}(k,0)$-modules. That is,
\begin{equation}\label{eq:dec-Heisenberg}
V(k,0) = \oplus_{\lambda\in Q} M_{\wh}(k,\lambda) \otimes
N_\lambda,
\end{equation}
\begin{equation}\label{eq:dec-Heisenberg1}
V(k,0)(\lambda)= M_{\wh}(k,\lambda) \otimes N_\lambda
\end{equation}
where
\begin{equation*}
N_\lambda = \{ v \in V(k,0)\,|\, h(m)v =\lambda(h)\delta_{m,0}v
\text{ for }  h\in \mathfrak h, m \ge 0\}
\end{equation*}
is the space of highest weight vectors with highest weight $\lambda$ for
$\wh.$

Note that $N(\g,k)=N_0$ is the commutant of
$V_{\wh}(k,0)$ in $V(k,0)$\cite{FZ}. The commutant $N(\g,k)$ is a vertex
operator algebra with the Virasoro vector $\omega =
\omega_{\mraff} - \omega_{\mathfrak h}$ whose central charge is
$\frac{kn(2n-1)}{k+n+\frac{1}{2}}-n.$


We let
\begin{equation}\label{eq:w3}
\begin{split}
\omega_{\alpha}
=\frac{1}{2k(k+2)}(-h_{\alpha}(-1)^{2}{\1}
+2ke_{\alpha}(-1)e_{-\alpha}(-1){\1}-kh_{\alpha}(-2){\1}),
\end{split}
\end{equation}

\begin{equation}\label{eq:w3'}
\begin{split}
\bar{\omega}_{\alpha}
=-h_{\alpha}(-1)^{2}{\1}
+4kx_{\frac{1}{2}\alpha}(-1)x_{-\frac{1}{2}\alpha}(-1){\1}-2kh_{\alpha}(-2){\1},
\end{split}
\end{equation}

\begin{equation}\label{eq:W3'}
\begin{split}
W_{\alpha}^3 &= k^2 h_\alpha(-3){\1} + 3 k
h_\alpha(-2)h_\alpha(-1){\1} + 2h_\alpha(-1)^3{\1}
-6k h_\alpha(-1)e_{\alpha}(-1)e_{-\alpha}(-1){\1}
\\
&+3k^2e_{\alpha}(-2)e_{-\alpha}(-1){\1}
-3k^2e_{\alpha}(-1)e_{-\alpha}(-2){\1},
\end{split}
\end{equation}

\begin{equation}\label{eq:W3'''}
\begin{split}
\bar{W}_{\alpha}^3 &= k^2 h_\alpha(-3){\1} + 3 k
h_\alpha(-2)h_\alpha(-1){\1} + h_\alpha(-1)^3{\1}
-6k h_\alpha(-1)x_{\frac{1}{2}\alpha}(-1)x_{-\frac{1}{2}\alpha}(-1){\1}\\&
+6k^2x_{\frac{1}{2}\alpha}(-2)x_{-\frac{1}{2}\alpha}(-1){\1}-6k^2x_{\frac{1}{2}\alpha}(-1)x_{-\frac{1}{2}\alpha}(-2){\1},\end{split}
\end{equation}

for $\alpha\in \Delta_{0(+)}^{L}$.

\begin{equation}\label{eq:W3}
\begin{split}
\omega_{\alpha}
=\frac{1}{8k(k+1)}( -2kh_\alpha(-2){\1} -h_\alpha(-1)^{2}{\1}
+4ke_{\alpha}(-1)e_{-\alpha}(-1){\1}),
\end{split}
\end{equation}

\begin{equation}\label{eq:W3''}
\begin{split}
W_{\alpha}^3 &=4k^2 h_\alpha(-3){\1} + 6 k
h_\alpha(-2)h_\alpha(-1){\1} + 2h_\alpha(-1)^3{\1}
-12k h_\alpha(-1)e_{\alpha}(-1)e_{-\alpha}(-1){\1}
\\
&+12k^2e_{\alpha}(-2)e_{-\alpha}(-1){\1}-12k^2e_{\alpha}(-1)e_{-\alpha}(-2){\1},
\end{split}
\end{equation}
for $\alpha\in \Delta_{0(+)}^{S}$.


The following theorem gives the generators of $N(\g,k)$.
\begin{thm}\label{generator2} The vertex operator
algebra $N(\g,k)$ is generated by dim$\mathfrak{g}-\mbox{dim}\mathfrak{h}$ vectors $\omega_{\alpha}$, $\bar{\omega}_{\alpha}$, $W_{\alpha}^3$, $\bar{W}_{\alpha}^3$ for $\alpha\in \Delta_{0(+)}^{L}$ and $\omega_{\alpha}$, $W_{\alpha}^3$ for $\alpha\in \Delta_{0(+)}^{S}$. That is, $N(\g,k)$  is generated by $N(osp(1|2),k_{\alpha})$ for $\alpha\in \Delta_{0(+)}^{L}$ and $N(sl_2,k_{\alpha})$ for  $\alpha\in \Delta_{0(+)}^{S}$, $k_{\alpha}=\frac{2}{\langle\alpha,\alpha\rangle}k$.

\end{thm}

\begin{proof} Note that $V(k,0)(0)= M_{\wh}(k,0)
\otimes N(\g,k).$  Firstly we prove that
$V(k,0)(0)$ is generated by vectors  $\alpha_{i}(-1)\1$, $\omega_{\alpha}$, $\bar{\omega}_{\alpha}$, $W_{\alpha}^3$, $\bar{W}_{\alpha}^3$ for $i=1,\cdots, n$, $\alpha\in \Delta_{0(+)}^{L}$ and $\omega_{\alpha}$, $W_{\alpha}^3$ for $\alpha\in \Delta_{0(+)}^{S}$. In fact, let $U$ be the vertex operator
subalgebra generated by  $h(-1)\1$ for $h\in\h$, $\omega_{\alpha}$, $\bar{\omega}_{\alpha}$, $W_{\alpha}^3$, $\bar{W}_{\alpha}^3$ for $\alpha\in \Delta_{0(+)}^{L}$ and $\omega_{\alpha}$, $W_{\alpha}^3$ for $\alpha\in \Delta_{0(+)}^{S}$. Then
$e_{-\alpha}(-1)e_{\alpha}(-1)\1\in U$, $x_{-\frac{1}{2}\alpha}(-1)x_{\frac{1}{2}\alpha}(-1)\1\in U$ and $\omega_{\mraff}\in U$. Moreover, from the expression
 of $W_{\alpha}^3$ and $\bar{W}_{\alpha}^3$, we see that $e_{-\alpha}(-1)e_{\alpha}(-2)\1 -
e_{-\alpha}(-2)e_{\alpha}(-1)\1\in U$, $x_{-\frac{1}{2}\alpha}(-1)x_{\frac{1}{2}\alpha}(-2)\1 -
x_{-\frac{1}{2}\alpha}(-2)x_{\frac{1}{2}\alpha}(-1)\1\in U$. Set
$L_{\mraff}(n)=({\omega_{\mraff}})_{n+1}$. A direct calculation shows that
\begin{equation*}
[L_{\mraff}(m), a(n)]=-na(m+n)
\end{equation*}
for $m,n\in\Z, a\in \g$. Thus,
\begin{equation*}
\begin{split}
L_{\mraff}(-1)e_{-\alpha}(-1)e_{\alpha}(-1)\1=e_{-\alpha}(-2)e_{\alpha}(-1)\1+e_{-\alpha}(-1)e_{\alpha}(-2)\1\in
U.
\end{split}
\end{equation*}
\begin{equation*}
\begin{split}
L_{\mraff}(-1)x_{-\frac{1}{2}\alpha}(-1)x_{\frac{1}{2}\alpha}(-1)\1=x_{-\frac{1}{2}\alpha}(-2)x_{\frac{1}{2}\alpha}(-1)\1+x_{-\frac{1}{2}\alpha}(-1)x_{\frac{1}{2}\alpha}(-2)\1\in
U.
\end{split}
\end{equation*}
Since $e_{-\alpha}(-1)e_{\alpha}(-2)\1 -
e_{-\alpha}(-2)e_{\alpha}(-1)\1\in U$, $x_{-\frac{1}{2}\alpha}(-1)x_{\frac{1}{2}\alpha}(-2)\1 -
x_{-\frac{1}{2}\alpha}(-2)x_{\frac{1}{2}\alpha}(-1)\1\in U$, we get $e_{-\alpha}(-2)e_{\alpha}(-1)\1 \in U$,
$x_{-\frac{1}{2}\alpha}(-2)x_{\frac{1}{2}\alpha}(-1)\1 \in U$, thus by Theorem \ref{generator1}, $U$ is equal to
$V(k,0)(0)$.

Next we show that $\omega_{\alpha}$, $\bar{\omega}_{\alpha}$, $W_{\alpha}^3$, $\bar{W}_{\alpha}^3\in N(\g,k)$ for $\alpha\in \Delta_{0(+)}^{L}$ and $\omega_{\alpha}$, $W_{\alpha}^3\in N(\g,k)$ for $\alpha\in \Delta_{0(+)}^{S}$. Since $\la h_\alpha,h_\alpha\ra \ne 0,$ we
have decomposition $\mathfrak h=\C h_\alpha \oplus (\C
h_\alpha)^{\bot}$, where $(\C h_\alpha)^{\bot}$ is the orthogonal
complement of $\C h_\alpha$ with respect to $\la,\ra.$ By direct calculations, we know that
$h_\alpha(n)\omega_{\alpha}=h_\alpha(n)\bar{\omega}_{\alpha}=h_\alpha(n)W_{\alpha}^3=h_\alpha(n)\bar{W}_{\alpha}^3=0$ for $n\geq
0,$ $\alpha\in \Delta_{0(+)}^{L}$, and $h_\alpha(n)\omega_{\alpha}=h_\alpha(n)W_{\alpha}^3=0$ for $n\geq
0,$ $\alpha\in \Delta_{0(+)}^{S}$. If $u\in (\C h_\alpha)^{\bot},$ we have $u(n)\omega_{\alpha}=u(n)\bar{\omega}_{\alpha}=u(n)W_{\alpha}^3=u(n)\bar{W}_{\alpha}^3=0$ for $n\geq
0,$ $\alpha\in \Delta_{0(+)}^{L}$, and $u(n)\omega_{\alpha}=u(n)W_{\alpha}^3=0$ for $n\geq
0,$ $\alpha\in \Delta_{0(+)}^{S}$.This
shows that $\omega_{\alpha}$, $\bar{\omega}_{\alpha}$, $W_{\alpha}^3$, $\bar{W}_{\alpha}^3\in N(\g,k)$ for $\alpha\in \Delta_{0(+)}^{L}$ and $\omega_{\alpha}$, $W_{\alpha}^3\in N(\g,k)$ for $\alpha\in \Delta_{0(+)}^{S}$.

Since $Y(u,z_1)Y(v,z_2 )=Y(v,z_2)Y(u,z_1)$ for $u \in
M_{\wh}(k,0)$, $v \in N(\g,k)$ and $V(k,0)(0)= M_{\wh}(k,0)
\otimes N(\g,k),$ $h(-1)\1 \in M_{\wh}(k,0)$ for $h\in \h$,
$\omega_{\alpha}$, $\bar{\omega}_{\alpha}$, $W_{\alpha}^3$, $\bar{W}_{\alpha}^3\in N(\g,k)$ for $\alpha\in \Delta_{0(+)}^{L}$ and $\omega_{\alpha}$, $W_{\alpha}^3\in N(\g,k)$ for $\alpha\in \Delta_{0(+)}^{S}$, it follows that $N(\g,k)$ is generated by
 $\omega_{\alpha}$, $\bar{\omega}_{\alpha}$, $W_{\alpha}^3$, $\bar{W}_{\alpha}^3$ for $\alpha\in \Delta_{0(+)}^{L}$ and $\omega_{\alpha}$, $W_{\alpha}^3$ for $\alpha\in \Delta_{0(+)}^{S}$.
\end{proof}

\begin{rmk} We want to point out that generators $\omega_{\alpha}$ for $\alpha\in \Delta_{0(+)}$ are Virasoro elements, but $\bar{\omega}_{\alpha}$ for $\alpha\in \Delta_{0(+)}^{L}$ are not Virasoro elements. Notice that the generators of $N(\g,k)$ are all in even part of the affine vertex operator superalgebra $V(k,0)$, thus $N(\g,k)$ is a vertex operator algebra. Moreover, the vertex operator algebra $N(\g,k)$ and its quotient $K(\g,k)$
are of moonshine type. That is, their weight zero subspaces are 1-dimensional
and weight one subspaces are zero.
\end{rmk}


\section{Parafermion vertex operator algebras $K(\g,k)$
}\label{Sect:maximal-ideal-tI}
\def\theequation{4.\arabic{equation}}
\setcounter{equation}{0}

The vertex operator superalgebra $V(k,0)$ has a
unique maximal ideal $\J$ generated by a weight $k+1$ vector
$e_{\theta}(-1)^{k+1}\1$ \cite{GS}, where $\theta$ is the highest
root of $\g$, and $e_{\theta}\in \g_{0}$, $\g_{0}$ is the even part of $\g$. The quotient vertex operator superalgebra $L(k,0) =
V(k,0)/\J$ is a simple, rational and $C_2$-cofinite  vertex operator algebra
associated to affine Lie algebra $\widehat{\g}$ \cite{AL},\cite{CL}. Moreover, the
Heisenberg vertex operator algebra $V_{\wh}(k,0)$ generated by
$h(-1)\1$ for $h\in \mathfrak h$ is a simple subalgebra of
$L(k,0)$ and $L(k,0)$ is a completely reducible
$V_{\wh}(k,0)$-module. We have a decomposition
\begin{equation}
L(k,0) = \oplus_{\lambda\in Q} M_{\wh}(k,\lambda) \otimes
K_\lambda
\end{equation}
as modules for $V_{\wh}(k,0)$, where
\begin{equation*}
K_\lambda = \{v \in L(k,0)\,|\, h(m)v =\lambda(h)\delta_{m,0}v
\text{ for }\; h\in {\mathfrak h},
 m \ge 0\}.
\end{equation*}
Set $K(\g,k)=K_0.$ Then $K(\g,k)$ is the commutant of
$V_{\wh}(k,0)$ in $L(k,0)$ and is called the parafermion vertex
operator algebra associated to the integrable highest weight
module $L(k,0)$ for $\widehat{\g}.$ Since $K(\g,k)$ is the extension of the rational and $C_2$-cofinite vertex operator algebra $K(sp(2n),k)$, $K(\g,k)$ is rational and
$C_2$-cofinite \cite{HKL}.

As a $V_{\wh}(k,0)$-module, $\J$ is completely reducible. From
\eqref{eq:dec-Heisenberg},
\begin{equation*}
\J = \oplus_{\lambda\in Q} M_{\wh}(k,\lambda) \otimes (\J \cap
N_\lambda).
\end{equation*}
In particular, $\tI = \J \cap N(\g,k)$ is an ideal of $N(\g,k)$
and $K(\g,k) \cong N(\g,k)/\tI$. Following the proof as
\cite[Lemma 3.1]{DLY2}, we know that $\tI$ is the unique maximal
ideal of $N(\g,k).$ Thus $K(\g,k)$ is a simple vertex operator
algebra. We still use $\omega_{\mraff}$, $\omega_{\mathfrak h}$,
$\omega_{\alpha}$,$\bar{\omega}_{\alpha}$, $W_{\alpha}^3$, $\bar{W}_{\alpha}^3$ to denote their
images in $L(k,0) = V(k,0)/\J$.
The following result follows from Theorem
\ref{generator2}.

\begin{thm}\label{generator3} The simple vertex operator algebra $K(\g,k)$
is generated by $\omega_{\alpha}$, $\bar{\omega}_{\alpha}$, $W_{\alpha}^3$, $\bar{W}_{\alpha}^3$ for $\alpha\in \Delta_{0(+)}^{L}$ and $\omega_{\alpha}$, $W_{\alpha}^3$ for $\alpha\in \Delta_{0(+)}^{S}$.
\end{thm}

Now we characterize the ideal $\tI$ of $N(\g,k)$. The vector
$x_{\theta}(-1)^{k+1}\1\notin N(\g,k)$. From \cite[Theorem
3.2]{DLY2} we know that
$h_\theta(n)x_{-\theta}(0)^{k+1}x_{\theta}(-1)^{k+1}\1=0$ for
$n\geq 0.$ It is clear that if $h\in \mathfrak h$ satisfies $\la
h_\theta,h\ra=0$, then
$h(n)x_{-\theta}(0)^{k+1}x_{\theta}(-1)^{k+1}\1=0$ for $n\geq 0.$
So we have the following result.

\begin{lem}\label{l1} $e_{-\theta}(0)^{k+1}e_{\theta}(-1)^{k+1}\1\in \tI.$
\end{lem}
Furthermore, similar to the proof of \cite[Theorem 4.2 (1)]{DLWY}, we have:

\begin{prop}\label{Conj:ideal-generator}
 The maximal ideal $\tI$ of $N(\g,k)$ is generated by $e_{-\theta}(0)^{k+1}e_{\theta}(-1)^{k+1}\1$.

\end{prop}

For $\alpha\in \Delta_{0(+)}^{L}$, we let $\widehat{P}_{\alpha}$ be the vertex
operator subalgebra of $N(\g,k)$ generated by $\omega_{\alpha}$, $\bar{\omega}_{\alpha}$, $W_{\alpha}^3$, $\bar{W}_{\alpha}^3$ and let $P_{\alpha}$ be the vertex
operator subalgebra of $K(\g,k)$ generated by $\omega_{\alpha}$, $\bar{\omega}_{\alpha}$, $W_{\alpha}^3$, $\bar{W}_{\alpha}^3$.  Then $P_\alpha$ is a quotient of
$\widehat{P}_{\alpha}.$ For $\alpha\in \Delta_{0(+)}^{S}$, we let $\widehat{P}^{'}_{\alpha}$ be the vertex
operator subalgebra of $N(\g,k)$ generated by $\omega_{\alpha}$, $W_{\alpha}^3$ and let $P^{'}_{\alpha}$ be the vertex
operator subalgebra of $K(\g,k)$ generated by $\omega_{\alpha}$, $W_{\alpha}^3$. Then $P^{'}_\alpha$ is a quotient of
$\widehat{P}^{'}_{\alpha}.$ Next we prove that
both $P_\alpha$ and $P^{'}_\alpha$ are simple vertex operator algebras. We have:



\begin{prop}\label{plast} For any $\alpha\in \Delta_{0(+)}^{L},$ the vertex operator
subalgebra $P_{\alpha}$ of $K(\g,k)$ is a simple vertex operator algebra
isomorphic to  the parafermion vertex operator algebra $K(osp(1|2),k)$. Let $\alpha\in \Delta_{0(+)}^{S}.$ Then the vertex operator subalgebra $P^{'}_{\alpha}$
of $K(\g,k)$ is a simple vertex operator algebra isomorphic to
the parafermion vertex operator algebra $K(sl_2,2k)$.
\end{prop}


\begin{proof} By \cite[Theorem 4.2]{DLWY}, we only
need to prove that
$$e_{-\a}(0)^{k_\a+1}e_{\a}(-1)^{k_\a+1}\1\in \tI,$$
where $k_\a=k$ if $\a\in \Delta_{0(+)}^{L}$ and $k_\a=2k$ if $\a\in \Delta_{0(+)}^{S}$.
Since $e_{\a}(-1)$ is locally nilpotent on $L(k,0),$ $L(k,0)$ is an integrable module for $\widehat{\g^\a}=\g^\a \otimes \C[t,t^{-1}] \oplus \C K$,
where $\g^{\a}=\C e_{\a}+\C h_{\alpha}+\C e_{-\alpha}$ is isomorphic
to $sl_2$. In particular, the
vertex operator subalgebra $U$ of $L(k,0)$ generated by $\g^\a$ is
an integrable highest weight module of $\widehat{\g^\a}$. That is, $U$ is isomorphic to
$L(k_\a,0)$ associated to the affine algebra $\widehat{\g^\a}.$ Thus we have $e_{\a}(-1)^{k_\a+1}\1\in \J.$ It follows that $e_{-\a}(0)^{k_\a+1}e_{\a}(-1)^{k_\a+1}\1\in
\tI$.
\end{proof}

\begin{rmk} We see that the building blocks of parafermion vertex operator algebras $K(osp(1|2n),k)$ are $K(osp(1|2),k)$ and $K(sl_2,2k)$. The structural and
representation theory for $K(sl_2,k)$ are studied in \cite{DLWY}, \cite{DW2}, \cite{ALY}, \cite{JW1}, \cite{JW2} etc., and the
representation theory for $K(osp(1|2),k)$ are studied in \cite{CFK}. These may shed light on the study of the
representation theory for rational parafermion vertex operator algebras $K(osp(1|2n),k)$.
\end{rmk}

\end{document}